\RequirePackage{amsmath}
\documentclass[smallextended,numbook,runningheads]{svjour3}     
\smartqed
\usepackage[T1]{fontenc}
\usepackage[latin1]{inputenc}
\usepackage[english]{babel}
\usepackage{mathptmx}
\usepackage{graphicx,subfigure,cite}
\usepackage{amsmath}
\usepackage{amssymb}
\usepackage{microtype}
\usepackage{enumerate}
\usepackage{amssymb,amsfonts}
\usepackage{silence}
\usepackage{hyperref}
\usepackage[capitalise]{cleveref}
\usepackage{todonotes}

\usepackage{tikzexternal}
\newtheorem{assumption}{Assumption}[section]
\newcommand{\tsign}[2]{\ensuremath{\stackrel{\text{#1}}#2}}
\newcommand{\teq}[1]{\tsign{#1}{=}}

\newfont{\bbold}{bbold10 scaled\magstep1}
\newcommand{\Id}{1\kern-0.25em{\rm l}}
\newcommand{\D}{\displaystyle}

\newcommand{\real}{\mathbb{R}}

\providecommand{\E}{\operatorname{E}}
\providecommand{\N}{\mathbb{N}}

\newcommand{\bO}{\ensuremath{\mathcal{O}}}
\newcommand{\hmax}{{\tilde{h}}}
\newcommand{\Ih}{I^{\hmax}}
\newcommand{\xvar}{{x}}
\newcommand{\varphit}{{\varphi_{{t}}}}

\newcommand{\Mpure}{\mu}
\newcommand{\M}{{\Mpure_{{t}}}}

\newcommand{\diff}{\,\mathrm{d}}

\newcommand{\dt}{\diff t}
\newcommand{\dX}{\diff X}
\newcommand{\dW}{\diff W}
\newcommand{\dM}{\diff\M}
\newcommand{\dMpure}{\diff\Mpure}
\newcommand{\DM}[1]{{\Delta_{{#1}}\Mpure}}

\newcommand{\R}{\ensuremath{\mathbb{R}}}
\newcommand{\numweightpure}{\Phi}
\newcommand{\numweight}{{\numweightpure_{{t}}}}
\newcommand{\hnumweight}{\hat{\numweightpure}}
\newcommand{\numstageweightpure}{\eta}
\newcommand{\numstageweight}{{\numstageweightpure_{{t}}}}
\newcommand{\hnumstageweight}{\hat{\numstageweightpure}}
\newcommand{\pd}{\ensuremath{p_d}}
\newcommand{\pmu}{p_\Mpure}
\newcommand{\rh}[1]{\rho(#1)}
\usepackage{xcolor,todonotes}
\def\makeheadbox{{%
\hbox to0pt{\vbox{\baselineskip=10dd\hrule\hbox
to\hsize{\vrule\kern3pt\vbox{\kern3pt
\hbox{\textbf{To appear in BIT Numerical Mathematics}}
\hbox{DOI: \href{http://dx.doi.org/10.1007/s10543-016-0619-8}{10.1007/s10543-016-0619-8}}
\kern3pt}\hfil\kern3pt\vrule}\hrule}%
\hss}}}
\begin{document}

\title{Cheap arbitrary high order methods for single integrand SDEs}

\author{Kristian Debrabant \and Anne Kv{\ae}rn{\o} 
}

\institute{Kristian Debrabant \at
  Department of Mathematics and Computer Science, University of Southern Denmark, 5230 Odense M, Denmark\\
  \email{debrabant@imada.sdu.dk}           
  \and
  Anne Kv{\ae}rn{\o} \at
  Department of Mathematical Sciences, Norwegian University of Science and Technology, 7491 Trondheim, Norway\\
  \email{anne.kvarno@math.ntnu.no}
}

\date{}
\maketitle

\begin{abstract}For a particular class of Stratonovich SDE problems, here denoted as single
  integrand SDEs, we prove that by applying a deterministic Runge--Kutta method of order $p_d$ we obtain methods converging in the mean-square and weak sense with order $\lfloor p_d/2\rfloor$. The reason is that
     the B-series of the exact solution and numerical approximation are, due to the single integrand and the usual rules of calculus holding for Stratonovich integration, similar to the ODE case. The only difference is that integration with respect to time is replaced by integration with respect to the measure induced by the single integrand SDE.
\end{abstract}
\section{Introduction}
\label{sec:intro}
In this paper we consider a particular class of Stratonovich stochastic differential equations (SDEs), single integrand SDEs, given by
\begin{equation}
  \dX = \lambda f(X)\dt + \sigma f(X)\circ \dW, \qquad {X}(t_0)=x_0
  \label{eq:sde}
\end{equation}
{where  $W(t)$ is a Wiener process},
$\lambda\in\{0,1\}$ and $\sigma \in \real$ is a given constant. {We assume the
  coefficient $f:\real^d \rightarrow \real^d$ to be differentiable
and that $f$ and $f'f$ satisfy a Lipschitz condition such that there exists a unique solution \cite{kloeden99nso}.}
The case
$\lambda=0$ covers Stratonovich SDEs without drift term. For $\lambda=1$, this class of methods arises frequently in applications, especially when modelling phenomena by ordinary differential equations and then introducing
{multiplicative} random fluctuations of uncertainties in time.

All results in this paper {hold as well for single integrand SDEs with} multidimensional {Wiener process},
\begin{equation}
  \dX = \lambda f(X)\dt + \sum_{i=1}^m\sigma_i f(X)\circ \dW_i, \qquad {X}(t_0)=x_0,
  \label{eq:sdemultidim}
\end{equation}
{as this case can be reduced to \eqref{eq:sde} using $\sigma  := \sqrt{\sum_{i=1}^m\sigma_i^2}$ and the Wiener process $W := \frac1\sigma\sum_{i=1}^m\sigma_iW_i$.}

Some well known examples for single integrand SDEs are the SDE describing fatigue cracking
\cite{kloeden99nso,sobczyk87smf}, the Kubo oscillator \cite{milstein04snf,cohen12otn}, the stochastic Van der Pol equation \cite{tian01itm} and certain
stochastic Hamiltonian problems, see \cite{cohen14epi}.

We are interested in solving \eqref{eq:sde} on the interval $I=[t_0,T]${.}
The equation can also be written in integral form as
\begin{equation}
  X(t) = x_0 + \int_{{t_0}}^t f(X(s))\circ \dMpure(s),
\qquad\Mpure({s}){:}=\lambda{s} + \sigma W(s).
  \label{eq:sde_int}
\end{equation}
Let a discretization ${\Ih}{:} = \{t_0, t_1, \ldots, t_N\}$ with
$t_0 < t_1 < \ldots < t_N =T$ of the time interval $I$ with step
sizes $h_n{:} = t_{n+1}-t_n$ for $n=0,1, \ldots, N-1$ {and maximal step size $\hmax:=\max_{n=1}^{N-1}h_n$} be given.

The SDE is solved by an $s$-stage stochastic Runge--Kutta method defined by
\begin{subequations}\label{eq:SRK}
  \begin{align}
    H_i &= Y_n + \DM{t_n,h_n}\sum_{j=1}^s a_{ij} f(H_j), \label{rk:a} \\
    Y_{n+1} &= Y_n + \DM{t_n,h_n}\sum_{{i}=1}^s b_i f(H_i).
    \label{eq:rkb}
  \end{align}
\end{subequations}
{Here,
$\DM{t_n,h_n}=\Mpure_{t_n}(h_n)$
with
\[
\M(s):=\Mpure(t+s)-\Mpure(t)=\lambda s + \sigma (W(t+{s}){-W(t)}),
\]}
{and typically, the coefficients $a_{ij}$ and $b_i$ will be those of a known Runge--Kutta method for ordinary differential
equations.}
For a general SDE, this simple generalization will result in a strong order 1 method, at
the best \cite{burrage96hso,burrage06con}.
But for the single integrand SDEs the situation is far better, as stated in the main result of
this paper:
{
\begin{theorem}\label{th:main}
  The Runge--Kutta method  \eqref{eq:SRK} of deterministic order $\pd$ is of mean square
  as well as weak  order  $\pmu = \lfloor\pd/2 \rfloor$, under the conditions on
  $f$ specified  in \cref{ass:f}.
  For weak convergence, it suffices that $\DM{t,h}$ is chosen such that at least the first $2\pmu+1$ moments coincide
with those of $\M(h)$, and all the others are in $\bO(h^{\pmu+1})$.
\end{theorem}
}
Here {and in the following, the $\bO$ notation refers to the absolute value and $h\to0$, }
 and for $x\in\R$, $\lfloor x\rfloor$ denotes the largest integer not larger than $x$.
\Cref{th:main} will be proved in Section \ref{sec:mainproof}.  Before going into the details, let us
justify the result by a simple numerical experiment:
\begin{example}\label{ex:intro}
  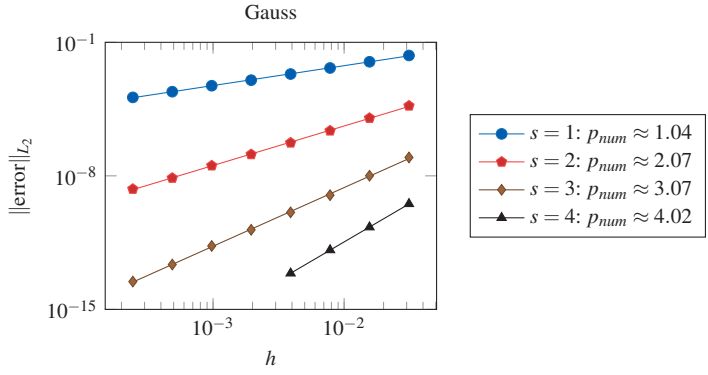
\begin{figure}[t]
    \centering
    \begin{tikzpicture}[baseline]
    \end{tikzpicture}
    \caption{The mean square error of Gauss methods applied to \eqref{eq:ex_intro}.}
    \label{fig:gauss_conv}
  \end{figure}
  We apply the $s$-stage Gauss method of deterministic order $2s$ {\cite{butcher08nmf}} to the SDE {\cite{kloeden99nso}}
  \begin{equation}
    \dX = \sqrt{1+X^2}\dt + \sigma \sqrt{1+X^2}\circ \dW, \qquad X(0)= 0
    \label{eq:ex_intro}
  \end{equation}
with the exact solution $X(t) = \sinh{(t+\sigma W(t))}$, with $\sigma=0.8$.
  The solution is approximated on the interval $[0,1]$ with step sizes $2^{-12}-2^{-5}$ and the sample
  average of $M=10,000$ independent simulated realizations of the absolute error is calculated in order
  to estimate the expectation. The results at $t=1$ are presented in Figure \ref{fig:gauss_conv}.
\end{example}
The outline of this paper is as follows: In section \ref{sec:prelim} some results on convergence and consistency are recalled.
In section \ref{sec:Bseries}, we will show that the B-series of the exact and the numerical solution are exactly
as in the ODE case \cite{butcher08nmf}, with {the exception that} integration is now performed with
respect to $\mu$ instead of $h$. This is due to the following lemma:
\begin{lemma}\label{th:integration}For all $h\geq0$ it holds
  \[
    \int_{0}^{h}\M^k\circ\dM=\frac1{k+1}\M(h)^{k+1}.
  \]
\end{lemma}
To deduce the mean square and the weak order of the approximations, the following lemma
is pivotal:
\begin{lemma}\label{th:mumoments}
  For $n\in\N$ it holds that $\E\M(h)^n=\begin{cases}
    \bO(h^{\frac{n}2})&:\text{ if $n$ is even},\\
    \bO(h^{\frac{n+1}2})&:\text{ if $n$ is odd}.
  \end{cases}$
\end{lemma}
The
proofs of \cref{th:main,th:integration,th:mumoments} are given in section \ref{sec:mainproof}. Finally, more numerical experiments
justifying the theoretical results are given in section \ref{sec:numexp}.
\section{Convergence and consistency}\label{sec:prelim}
Here we will give the definitions of both weak and strong convergence
and results which relate convergence to consistency.

Let $C_P^l(\real^d, \real^{\hat{d}})$ denote the space of all $g \in
C^l(\real^d,\real^{\hat{d}})$ fulfilling a polynomial growth condition \cite{kloeden99nso} and $\Ih$ be the discretized time interval defined above.
\begin{definition}
  A time discrete approximation $Y=(Y(t))_{t \in {\Ih}}$ converges weakly with order $p$ to $X$ at time $t \in {\Ih}$ as {the maximum step size}
  $\hmax\rightarrow 0$ if for each $g \in
  C_P^{2(p+1)}(\real^d, \real)$ there exist a constant $C_g$
  and a finite $\delta_0 > 0$ such that
  \begin{equation*}
    | \E(g(Y(t))) - \E(g(X(t))) | \leq C_g \, \hmax^p
  \end{equation*}
holds for each $\hmax\in \, ]0,\delta_0[\,$.
\end{definition}
Now, let $le_g(h;t,x)$ be the weak local error of the method starting
at the point $(t,x)$ with respect to the functional $g$ and step size
$h$, i.\,e.
\[
  le_g(h;t,x)= \E \big(g(Y(t+h))-g(X(t+h))|Y(t)=X(t)=x\big).
\]
The following theorem due to Milstein \cite{milstein95nio}, which
holds also in the case of general one step methods, shows that, as in
the deterministic case, consistency implies convergence:\pagebreak[2]
\begin{theorem}\label{th:weakconsimplconv}
  Suppose the following conditions hold:
  \begin{itemize}
    \item \label{St-lg-cond1} The {integrand} $f$ {of \eqref{eq:sde_int}} is differentiable, and $f$ and $f'f$ satisfy a Lipschitz condition and belong to $C_P^{2(p+1)}(\real^d,
        \real^d)$.
    \item \label{St-lg-cond2} For sufficiently large $r$ (see, e.g.,
      \cite{milstein95nio} for details) the moments
      $\E(\|Y(t_n)\|^{2r})$ exist for $t_n \in {\Ih}$ and are
      uniformly bounded with respect to $N$ and $n=0,1, \ldots, N$.
    \item \label{St-lg-cond3} For all $g\in
      C_P^{2(p+1)}(\real^d, \real)$ there exists a $K \in
      C_P^0(\real^d, \real)$ such that
      \begin{eqnarray*}
	|le_g(h;t,x)| &\leq& K(x) \, h^{p+1}
      \end{eqnarray*}
      is valid for $x \in \real^d$ and $t, t+h \in I$, i.\,e., the
      approximation is weak consistent of order $p$.
  \end{itemize}
  Then the method \eqref{eq:SRK} is convergent of order $p$ in the
  sense of weak approximation.
\end{theorem}
Whereas weak approximation methods are used to estimate the expectation of functionals of the solution, strong approximation methods approach the solution
path-wise.
\begin{definition}
  A time discrete approximation $Y=(Y(t))_{t \in {\Ih}}$ converges
  strongly respectively in the mean square with order $p$ to $X$ at time $t \in {\Ih}$ as
  {the maximum step size} $\hmax\rightarrow 0$ if there {exist} a constant
  $C$
  and a finite $\delta_0 > 0$ such that
  \begin{equation*}
    \E\|Y(t) - X(t)\|\leq C \, \hmax^p
    \qquad
    \text{respectively}
    \qquad
    \sqrt{\E(\|Y(t) - X(t)\|^2)}\leq C \, \hmax^p
  \end{equation*}
holds for each $\hmax \in \, ]0,\delta_0[\,$.
\end{definition}
In this article we will consider convergence in the mean
square sense. But by Jensen's inequality we have
\[
  (\E\|Y(t) - X(t)\|)^2\leq\E(\|Y(t) - X(t)\|^2),
\]
so mean square convergence implies strong convergence of the same order.

Now, let $le^m(h;t,x)$ respectively $le^{ms}(h;t,x)$ be the mean
respectively mean square local error of the method starting at the
point $(t,x)$ with respect to the step size $h$, i.\,e.
\begin{eqnarray*}
  le^m(h;t,x)&=& \E \big(Y(t+h)-X(t+h)|Y(t)=X(t)=x\big),\\
  le^{ms}(h;t,x)&=&\sqrt{\E \big((Y(t+h)-X(t+h))^2|Y(t)=X(t)=x\big)}.
\end{eqnarray*}
The following theorem due to Milstein \cite{milstein95nio}, which holds
also in the case of general one step methods, shows that in the mean
square convergence case we obtain order $p$ if the mean local error is
consistent of order $p$ and the mean square local error is consistent
of order $p-\frac12$.
\begin{theorem}\label{th:strongconsimplconv}
  Suppose the following conditions hold:
  \begin{itemize}
    \item {The} {integrand} $f$ {of \eqref{eq:sde_int}} is differentiable, and $f$ and $f'f$ satisfy a Lipschitz condition.
    \item There exists a constant $K$ independent of $h$
      such that
      \[
	\|le^m(h;t,x)\|\leq K\sqrt{1+\|x\|^2}\,h^{p_1},\qquad
	le^{ms}(h;t,x) \leq K\sqrt{1+\|x\|^2}\, h^{p+\frac12}
      \]
      with $p\geq0$, $p_1\geq p+1$ is valid for $x \in \real^d$ and
      $t, t+h \in {I}$, i.\,e., the approximation is consistent in
      the mean of order $p_1-1\geq p$ and in the mean square of
      order $p-\frac12$.
  \end{itemize}
  Then the SRK method (\ref{eq:SRK}) is convergent of order $p$ in
  the sense of mean square approximation.
\end{theorem}
\section{B-series and rooted trees}\label{sec:Bseries}
{In order to apply  \cref{th:weakconsimplconv,th:strongconsimplconv} we will now use B--series and rooted tree theory to study the
  order of the local errors of the method \eqref{eq:SRK}.}  B-series for deterministic ODEs were introduced by Butcher \cite{butcher63cft}. Today such series appear as a fundamental tool to do local error analysis on a
wide range of problems. B-series for SDEs and their numerical solution by stochastic Runge--Kutta methods have been developed by Burrage and Burrage
\cite{burrage96hso,burrage00oco} to study strong convergence in the Stratonovich case, by Komori, Mitsui and Sugiura \cite{komori97rta} and Komori
\cite{komori07mrt} to study weak convergence in the Stratonovich case and by R\"{o}{\ss}ler \cite{roessler04ste,roessler06rta} to study weak convergence in both the
It\^{o} and the Stratonovich case. However, the distinction between the It\^{o} and the Stratonovich integrals
only depends on the definition of the integrals, not on how the B-series are constructed. Similarly, the distinction between weak and strong convergence only
depends on the definition of the local error. A uniform and self-contained theory for the construction of stochastic B-series for the exact solution of SDEs
and its numerical approximation by stochastic Runge--Kutta methods is given in \cite{debrabant08bao}. Based on the notation used there, we will now derive the
B-series for the exact solution and numerical approximation of single-integrand SDEs. Due to the single integrand we will, similar to the ODE case \cite{butcher08nmf}, only need non-colored trees in the expansion of the solution.

\begin{definition}[Trees]
  The set of rooted trees $T$ is recursively defined as follows:
  \begin{description}
    \item[a)] The empty tree $\emptyset$ and the graph $\bullet=[\emptyset]$ with only one vertex {belong} to $T$.
  \end{description}
  Let $\tau=[\tau_1,\tau_2,{\dots},\tau_{\kappa}]$ be the tree
  formed by joining the subtrees
  $\tau_1,\tau_2,{\dots},\tau_{\kappa}$ each by a single branch to a
  common root.
  \begin{description}
    \item[b)] If $\tau_1,\tau_2,{\dots},\tau_{\kappa} \in T$ then
      $\tau=[\tau_1,\tau_2,{\dots},\tau_{\kappa}] \in T$.
  \end{description}
\end{definition}

\begin{definition}[Elementary differentials]
  For a tree $\tau \in T$ the elementary differential is
  a mapping $F(\tau):\real^d \rightarrow \real^d$ defined
  recursively by
  \begin{description}
    \item[a)] $F(\emptyset)(\xvar)=\xvar$, \\ \mbox{}
    \item[b)] $F(\bullet_l)(\xvar)=f(\xvar)$, \\ \mbox{}
    \item[c)] If $\tau=[\tau_1,\tau_2,{\dots},\tau_{\kappa}] \in T\setminus\{\emptyset\}$
      then
      \[
	F(\tau)(\xvar)=f^{(\kappa)}(\xvar)
      \big(F(\tau_1)(\xvar),F(\tau_2)(\xvar),{\dots},F(\tau_{\kappa})(\xvar)\big)\]
  \end{description}
{where $\xvar\in\R^d$.}
\end{definition}

\begin{definition}[B-series]
  Consider a family $\{\phi(\tau)\}_{\tau\in T}$ of random variables satisfying
  \[
    \phi(\emptyset)\equiv1 \;\text{ and }\; \phi(\tau)(0)=0,\quad
    \forall \tau\in T \backslash \{\emptyset\}.
  \]
  A (stochastic) B-series is then
  a formal series of the form
  \[
    B(\phi,\xvar; h) = \sum_{\tau \in T}
    \alpha(\tau)\cdot\phi(\tau)(h)\cdot F(\tau)(\xvar),
  \]
  where
  $\alpha: T\rightarrow \mathbb{Q}$ is given by
  \begin{align*}
    \alpha(\emptyset)&=1,&\alpha(\bullet)&=1,
    &\alpha(\tau=[\tau_1,\cdots,\tau_{\kappa}])&=
    \frac{1}{r_1!r_2!\cdots r_{q}! } \prod_{j=1}^{\kappa} \alpha(\tau_j),
  \end{align*}
  where $r_1,r_2,{\dots},r_{q}$ count equal trees among
  $\tau_1,\tau_2,{\dots},\tau_{\kappa}$.
\end{definition}

The next lemma proves that if $Y({t+}h)$ can be written as a B-series,
then $g(Y({t+}h))$ can be written as a similar series, where the sum is
taken over trees with a root of color $g$ and subtrees in $T$.  The
lemma is fundamental for deriving the B-series of the exact and the
numerical solution. It will also be used for deriving weak convergence
results.
\begin{lemma}[\cite{debrabant08bao}] \label{lem:f_y} If $Y({t+}h)=B(\phi, \xvar; h)$ is some
  B-series and $g\in C^{\infty}(\real^d,\real^{\hat{d}})$ then
  $g(Y({t+}h))$ can be written as a formal series of the form
  \begin{equation}
    \label{eq:f}
    g(Y({t+}h))=\sum_{u\in U_g} \beta(u)\cdot \psi_\phi(u)(h)\cdot G(u)(\xvar),
  \end{equation}
  where $U_g$ is a set of trees derived from $T$, by
  \begin{description}
    \item[a)] $[\emptyset]_g \in U_g$, and if
      $\tau_1,\tau_2,{\dots},\tau_{\kappa}\in T\setminus\{\emptyset\}$ then
      $[\tau_1,\tau_2,{\dots},\tau_{\kappa}]_g\in U_g$. \\ \mbox{}
    \item[b)] $G([\emptyset]_g)(\xvar)=g(\xvar)$ and \\
      $\mbox{}\;\, G(u=[\tau_1,{\dots},\tau_{\kappa}]_g)(\xvar) =
      g^{(\kappa)}(\xvar)
      \big(F(\tau_1)(\xvar),{\dots},F(\tau_{\kappa})(\xvar)\big)$. \\
      \mbox{}
    \item[c)] $\beta([\emptyset]_g)=1$ and $\D
      \beta(u=[\tau_1,{\dots},\tau_{\kappa}]_g)
      =\frac{1}{r_1!r_2!\cdots r_{q}!}\prod_{j=1}^{\kappa}
      \alpha(\tau_{{j}})$, \\where $r_1,r_2,{\dots},r_{q}$ count
      equal trees among $\tau_1,\tau_2,{\dots},\tau_{\kappa}$. \\
      \mbox{}
    \item[d)] $\psi_\phi([\emptyset]_g)\equiv1$ and
      $\psi_\phi(u=[\tau_1,{\dots},\tau_{\kappa}]_g)(h) =
      \prod_{j=1}^{\kappa} \phi(\tau_j)(h)$.
  \end{description}
\end{lemma}
We are now able to derive the B-series of the exact solution. Here and in the following,
$\rh{\tau}$ denotes the number of nodes in a tree $\tau$.

\begin{theorem}\label{th:BseriesexactsolutionODE}
  Let
  $\gamma: T\rightarrow\N$ be given by
  \begin{gather*}
    \gamma(\emptyset)=1,\qquad\gamma(\bullet)=1,\\
    \gamma([\tau_1,\dots,\tau_{\kappa}])=\rh{[\tau_1,\dots,\tau_{\kappa}]}\prod_{j=1}^{\kappa} \gamma(\tau_j).
  \end{gather*}
  Then the solution $X({t+}h)$ of \eqref{eq:sde} {starting at the
point $(t,\xvar)$} can be written as a B-series $B(\varphit,\xvar; h)$ with
  \begin{gather}\label{eq:explweightexsol}
    \varphit(\tau)(h)=\frac{\M(h)^{\rh{\tau}}}{\gamma(\tau)}\text{ for }\tau\in T.
  \end{gather}
\end{theorem}
\begin{proof}
  Write the exact solution as some B-series
  $X({t+}h)=B(\varphit,\xvar;h)$. By applying \cref{lem:f_y} to $f$ (in which case $U_g=T$) the integral
  form \eqref{eq:sde_int} of the SDE can be written as
  \[
    \sum_{\tau\in T} \alpha(\tau) \cdot \varphit(\tau)(h) \cdot
    F(\tau)(\xvar) = \xvar+ \sum_{\tau \in T} \alpha(\tau) \cdot \int_0^h
    \prod_{j=1}^\kappa \varphit(\tau_j)(s)\circ \dM(s) \cdot F(\tau)(\xvar).
  \]
  By comparing term by term we get
  \begin{gather*}
    \varphit(\emptyset)\equiv 1, \quad \varphit(\bullet)(h)=\M(h), \\
    \quad \varphit(\tau)(h) = \int_0^h \prod_{j=1}^{\kappa}\varphit(\tau_j)(s)\circ \dM(s) \quad
    \text{for}\quad \tau=[\tau_1,\dots,\tau_{\kappa}] \in T.
  \end{gather*}
  This proves the theorem for $\tau=\emptyset$ and $\tau=\bullet$. The rest is proved by induction
  on the height of $\tau$.  If
  $\tau=[\tau_1,\dots,\tau_{\kappa}]$ then
  \begin{multline*}
    \varphit(\tau)(h)=\int_0^h\prod_{j=1}^{\kappa}\varphit(\tau_j)(s)\circ \dM(s)
    =\int_0^h\prod_{j=1}^{\kappa}\frac{\M(s)^{\rh{\tau_j}}}{\gamma(\tau_j)}\circ \dM(s)\\
    \teq{\cref{th:integration}}\frac{\M(h)^{\rh{\tau}}}{\rh{\tau}}\prod_{j=1}^{\kappa}\frac1{\gamma(\tau_j)}=\frac{\M(h)^{\rh{\tau}}}{\gamma(\tau)}.
  \end{multline*}
  \qed
\end{proof}
For the numerical approximation \eqref{eq:SRK} the following result holds:
\begin{theorem} \label{thm:Bnum} The numerical solution $Y(t+h)$ {after one step with step size $h$ starting at the
point $(t,x)$} as
  well as the {corresponding} stage values $H_i$ can be written in terms of B-series
  \[ H_i = B(\numstageweight_i,\xvar; h), \qquad Y{(t+h)} =
  B(\numweight,\xvar; h)\] with $\numstageweight_i(\tau){(h)}={(}\DM{{t,t+h}}{)}^{\rh{\tau}}\hnumstageweight_i(\tau)$, where
  \begin{subequations} \label{eq:B_iter_H}
    \begin{gather}
      \hnumstageweight_i(\emptyset)=1, \quad
      \hnumstageweight_i(\bullet)=\sum_{j=1}^sa_{ij},\\
      \hnumstageweight_i(\tau)=\sum_{j=1}^sa_{ij}\prod_{k=1}^{\kappa}\hnumstageweight_j(\tau_k)\text{ if }\tau=[\tau_1,\dots,\tau_{\kappa}]
    \end{gather}
  \end{subequations}
  and $\numweight(\tau){(h)}={(}\DM{{t,t+h}}{)}^{\rh{\tau}}\hnumweight(\tau)$, where
  \begin{subequations} \label{eq:B_iter_Y}
    \begin{gather}
      \hnumweight(\emptyset)=1, \quad
      \hnumweight(\bullet)=\sum_{i=1}^sb_i,\\
      \hnumweight([\tau_1,\dots,\tau_{\kappa}])
      =\sum_{i=1}^sb_i\prod_{k=1}^{\kappa}\hnumstageweight_i(\tau_k).
    \end{gather}
  \end{subequations}
\end{theorem}
\begin{proof}
  Write $H_i$ as a B-series, that is
  \[H_i = \sum_{\tau\in T}\alpha(\tau)
  \numstageweight_i(\tau)(h)F(\tau)(\xvar).\] Use \eqref{rk:a} together with \cref{lem:f_y} to obtain
  \begin{align*}
    H_i&= \xvar +\DM{{t,t+h}}\sum_{j=1}^sa_{ij}f(H_j)\\
    &=\xvar +\DM{{t,t+h}}\sum_{j=1}^sa_{ij}f\big(\sum_{\tau\in T}\alpha(\tau)
    \numstageweight_j(\tau)(h)F(\tau)(\xvar)\big)\\
    &=\xvar+\DM{{t,t+h}}\sum_{j=1}^sa_{ij}\sum_{\tau\in T}\alpha(\tau)\psi_{\numstageweight_j}(\tau)(h)F(\tau)(\xvar)\\
    &=\xvar+\sum_{\tau\in T}\alpha(\tau)\left(\DM{{t,t+h}}\sum_{j=1}^sa_{ij}\psi_{\numstageweight_j}(\tau)(h)\right)F(\tau)(\xvar).
  \end{align*}
  Thus $\numstageweight_i(\tau)=\DM{{t,t+h}}\sum_{j=1}^sa_{ij}\psi_{\numstageweight_j}(\tau)+\begin{cases}
    1&: \tau=\emptyset\\
    0&: \tau\neq\emptyset
  \end{cases}$, proving \eqref{eq:B_iter_H}. Analogously,
  \begin{align*}
    Y{(t+h)}&=\xvar+\DM{{t,t+h}}\sum_{i=1}^sb_if(H_i)\\
    &=\xvar+\sum_{\tau\in T}\alpha(\tau)\left(\DM{{t,t+h}}\sum_{i=1}^sb_i\psi_{\numstageweight_i}(\tau)(h)\right)F(\tau)(\xvar),
  \end{align*}
  proving \eqref{eq:B_iter_Y}.
  \qed
\end{proof}
To decide the weak order we will also need the B-series of a function
$g$, evaluated at the exact and the numerical solution.
From \cref{th:BseriesexactsolutionODE,thm:Bnum,lem:f_y} we obtain
\[
  g(X({t+}h))=\sum_{u\in U_g}\beta(u)\cdot\psi_\varphit(u)(h)\cdot G(u)(\xvar),
\]
\[
  g(Y{(t+h)})=\sum_{u\in U_g}\beta(u)\cdot\psi_\numweight(u)(h)\cdot G(u)(\xvar),
\]
with \[\psi_\varphit([\emptyset]_g)\equiv1,\quad
\psi_\varphit(u=[\tau_1,{\dots},\tau_\kappa]_g)(h) =
\prod\limits_{j=1}^\kappa \varphit(\tau_j)(h)
\]
and \[\psi_\numweight([\emptyset]_g)\equiv1,\quad
\psi_\numweight(u=[\tau_1,{\dots},\tau_\kappa]_g)(h) =
\prod\limits_{j=1}^\kappa\numweight(\tau_j)(h).\]
So, for the weak local error it follows
\[
  le_g(h;t,x)=\sum_{u\in
  U_g}\beta(u)\cdot\E\left[\psi_\numweight(u)(h)-\psi_\varphit(u)(h)\right]\cdot
  G(u)(x).
\]
For the mean respectively mean square local error we obtain from \cref{th:BseriesexactsolutionODE,thm:Bnum}
\begin{eqnarray*}
  le^{ms}(h;t,x)&=&\sqrt{\E \big(\sum_{\tau\in
  T}\alpha(\tau)\cdot(\numweight(\tau){(h)}-\varphit(\tau){(h)})\cdot F(\tau)(x)\big)^2},\\
  le^m(h;t,x)&=&\sum_{\tau\in T}\alpha(\tau)\cdot
  \E\big(\numweight(\tau){(h)}-\varphit(\tau){(h)}\big)\cdot F(\tau)(x)
  .
\end{eqnarray*}
\section{Proofs of \texorpdfstring{\cref{th:main,th:integration,th:mumoments}}{Theorem \ref{th:main} and Lemmas \ref{th:integration} and \ref{th:mumoments}}} \label{sec:mainproof}
With all the B-series in place, we can now present the order
conditions for the weak and strong convergence. {For convenience, we first summarize the assumptions on $f$:
\begin{assumption}\label{ass:f}
Let $f\in C^{2(p+1)}(\real^d,
        \real^d)$ and $f$ and $f'f$ fulfill a Lipschitz condition.
 Further, assume
\begin{itemize}
    \item for mean-square convergence, that either
 \begin{itemize}
\item[*] all elementary differentials $F(\tau)$ fulfill a linear growth condition, or
\item[*] there exists a
constant $C$ such that $\|f'(y)\|\leq C\quad\forall y\in\real^d$ (which implies the global Lipschitz condition) and all necessary partial
derivatives exist \cite{burrage00oco},
\end{itemize}
\item respectively for weak convergence, that
    $f\in C_P^{2(p+1)}(\real^d,
        \real^d)$.
  \end{itemize}
\end{assumption}}

We have weak consistency of order $\pmu$ if and only if
\begin{equation}
    \label{eq:T_eq}
    \E\psi_\numweight(u)(h)=\E\psi_\varphit(u)(h)+\bO(h^{\pmu+1})\quad\forall u\in U_g
\end{equation}
where
$\rh{u=[\tau_{1},\cdots,\tau_{\kappa}]_{f}}=
\sum_{j=1}^{\kappa}\rh{\tau_{j}}$, and
mean square global order $\pmu$ if and only if
\begin{gather}\label{eq:StrOrdCond1}
{\E\big((\numweight(\tau)(h)-\varphit(\tau)(h))^2\big)=\bO(h^{2\pmu+1})}\quad\forall\tau\in T%
,\\
\label{eq:StrOrdCond2}
\E\numweight(\tau)(h)=\E\varphit(\tau)(h)+\bO(h^{\pmu+1})\quad\forall\tau\in T.%
\end{gather}
{Assume $\DM{t,h}=\M(h)$.}  Due to $\E\M(h)^{2\rh{\tau}}=\bO(h^{\rh{\tau}})$ by \cref{th:mumoments}, \eqref{eq:StrOrdCond1} is then by
\cref{th:BseriesexactsolutionODE,thm:Bnum} automatically fulfilled for all $\tau\in T$ with $\rh{\tau}\geq 2\pmu+1$, and satisfied for the remaining trees if and only if
\begin{equation}\label{eq:StrOrdCond1det}
\hnumweight(\tau)=\frac1{\gamma(\tau)}\quad\forall\tau\in T \text{ with }\rh{\tau}\leq2\pmu.
\end{equation}
Note that \eqref{eq:StrOrdCond1det} is just the condition that for the order $\pd$ of the Runge--Kutta method applied to a deterministic system ($\sigma=0$) it
holds $\pd=2\pmu$.

Similarly, \eqref{eq:StrOrdCond2} is automatically fulfilled for all $\tau\in T$ with
\[\rh{\tau}\geq \begin{cases}2\pmu+2&:\text{ for even }\rh{\tau}\\2\pmu+1&:\text{ for odd }\rh{\tau}\end{cases},\]
and satisfied for the remaining trees if and only if
\begin{equation}\label{eq:StrOrdCond2det}
\hnumweight(\tau)=\frac1{\gamma(\tau)}\quad\forall\tau\in T\text{ with }\rh{\tau}\leq\begin{cases}2\pmu+1&:\text{ for even
}\rh{\tau}\\2\pmu&:\text{ for odd }\rh{\tau}\end{cases}.
\end{equation}
Thus, the method will be mean-square consistent of order $\pmu$ if its deterministic order is
\begin{equation}
  \pd=\begin{cases}
    2\pmu&:\text{ if }\pmu\in\N\\
    2\pmu+1&:\text{ if }\pmu+\frac12\in\N\\
  \end{cases}
\end{equation}
or, vice versa, a method of deterministic order $\pd$ will converge with mean-square order $\lfloor\frac\pd2\rfloor$.

{Assume now that $\DM{t,h}$ is chosen such that at least the first $2\pmu+1$ moments coincide
with {those} of $\M(h)$, and all the others are in $\bO(h^{\pmu+1})$.}
Analogously to the discussion of \eqref{eq:StrOrdCond2}, \eqref{eq:T_eq} is automatically fulfilled for all $u\in U_g$ with
\[\rh{u}\geq \begin{cases}2\pmu+2&:\text{ for even }\rh{u}\\2\pmu+1&:\text{ for odd }\rh{u}\end{cases},\]
and satisfied for the remaining trees if and only if \eqref{eq:StrOrdCond2det} is fulfilled. Thus, we obtain that the weak order of the method equals its
mean-square order, which finishes the proof of \cref{th:main}.

Finally we present the proofs of \cref{th:integration,th:mumoments}.
\begin{proof}[\cref{th:integration}]
  As $\M$ is a semimartingale with continuous paths, it holds for $f\in C^2(\R,\R)$ that \[f(\M(h))-f(\M(0))=\int_0^hf'(\M(s))\circ\dM(s),\]
  which for $f(x)=x^{k+1}$ immediately gives the assertion.
  \qed
\end{proof}
\begin{proof}[\cref{th:mumoments}]
  The assertion follows from
  \[
    \E\M(h)^n=\sum_{i=0}^n\binom{n}{i}{\lambda^{n-i}}h^{n-i}{\sigma^i}\E\big({\big(}W({t+}h){-W(t)\big)}^i\big)
  \]
  and using that
  \[
    \E\big({\big(}W({t+}h){-W(t)\big)}^i\big)=h^{\frac{i}2}\begin{cases}
      0&:\text{ if $i$ is odd},\\
      (i-1)\cdot(i-3)\cdot\dots\cdot3\cdot1&\text{: otherwise.}
    \end{cases}
  \]\qed
\end{proof}
\section{Numerical experiments}\label{sec:numexp}
To verify the theoretical result, we solve three test problems by method \eqref{eq:SRK}, based
on some classes of well known deterministic Runge--Kutta methods. The first group consists of the Gauss
methods of deterministic order $\pd = 2s$, in which case the predicted stochastic order $\pmu=s$, as
already demonstrated in Example \ref{ex:intro}. Using the results from \cite{ma12sca, hong15poq} it
is straightforward to confirm that the Gauss methods preserve quadratic invariants.  Two of the test problems below have
such invariants, for these problems we also {demonstrate} this conservation property. The second group of
methods consists of the Radau IIA methods, these are
of deterministic order $p_d=2s-1$, thus $\pmu = s-1$. Finally, we consider three explicit
Runge--Kutta (ERK) methods: A third order, three stage method \cite[RK32, p.~95]{butcher08nmf}, the
classical fourth order Runge--Kutta method \cite[p.~138]{hairer10sodI} and the fifth order Fehlberg method
\cite[p.~177]{hairer10sodI}. In the following, they will be denoted by ERK3, ERK4 and ERK5,
respectively.

For the calculation of the numerical order $p_{num}$,
errors less than $10^{-14}$ have been ignored.
\begin{example}
  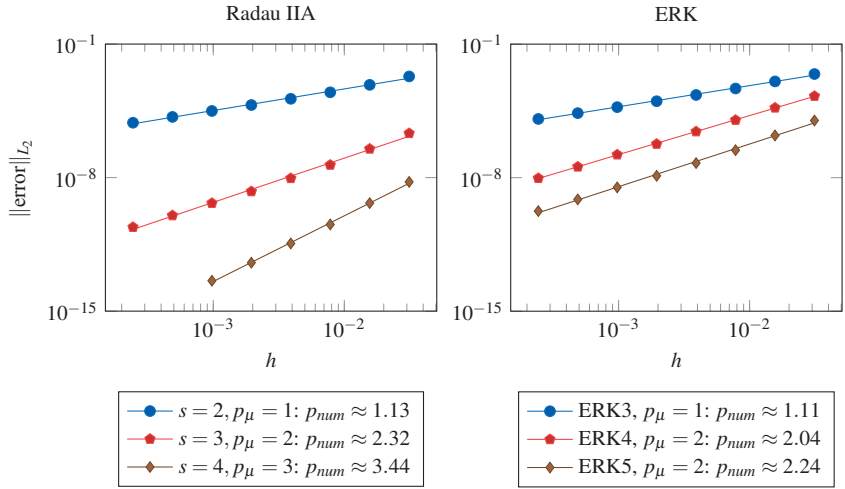
\begin{figure}
    \begin{tikzpicture}[baseline]
    \end{tikzpicture}
    \begin{tikzpicture}[baseline]
    \end{tikzpicture}
    \caption{Error plots for the Radau IIA and the ERK methods applied to \eqref{eq:ex_intro}. }
    \label{fig:KP1}
  \end{figure}
  This is a continuation of Example \ref{ex:intro}. The SDE \eqref{eq:ex_intro} is solved by the
  Radau IIA methods, as well as the three explicit methods. The number of independent simulations is
  still $M=10,000$. Convergence plots are given in Figure \ref{fig:KP1}, as well as the estimated
  order $p_{num}$. The order $p_{num}$ is slightly above $\pmu$, probably because the error of the
  deterministic part becomes more dominant for larger step sizes.
\end{example}

\begin{example}

  \begin{figure}
    \begin{tikzpicture}[baseline]
    \end{tikzpicture}
     \hspace*{\fill}
    \begin{tikzpicture}[baseline]
    \end{tikzpicture}
    \\[5mm]
    \begin{tikzpicture}[baseline]
    \end{tikzpicture}
     \hspace*{\fill}
    \begin{tikzpicture}[baseline]
    \end{tikzpicture}
    \caption{Error plots for the Kubo oscillator \eqref{eq:kubo} (\textit{top, and bottom left}).
    The solution points for one trajectory, using ERK5 (Fehlberg), Radau IIA ($s=3$) and Gauss ($s=2$)
    with $h=0.5$ for $t\leq 1000$ (\textit{bottom right}).}
    \label{fig:kubo}
  \end{figure}
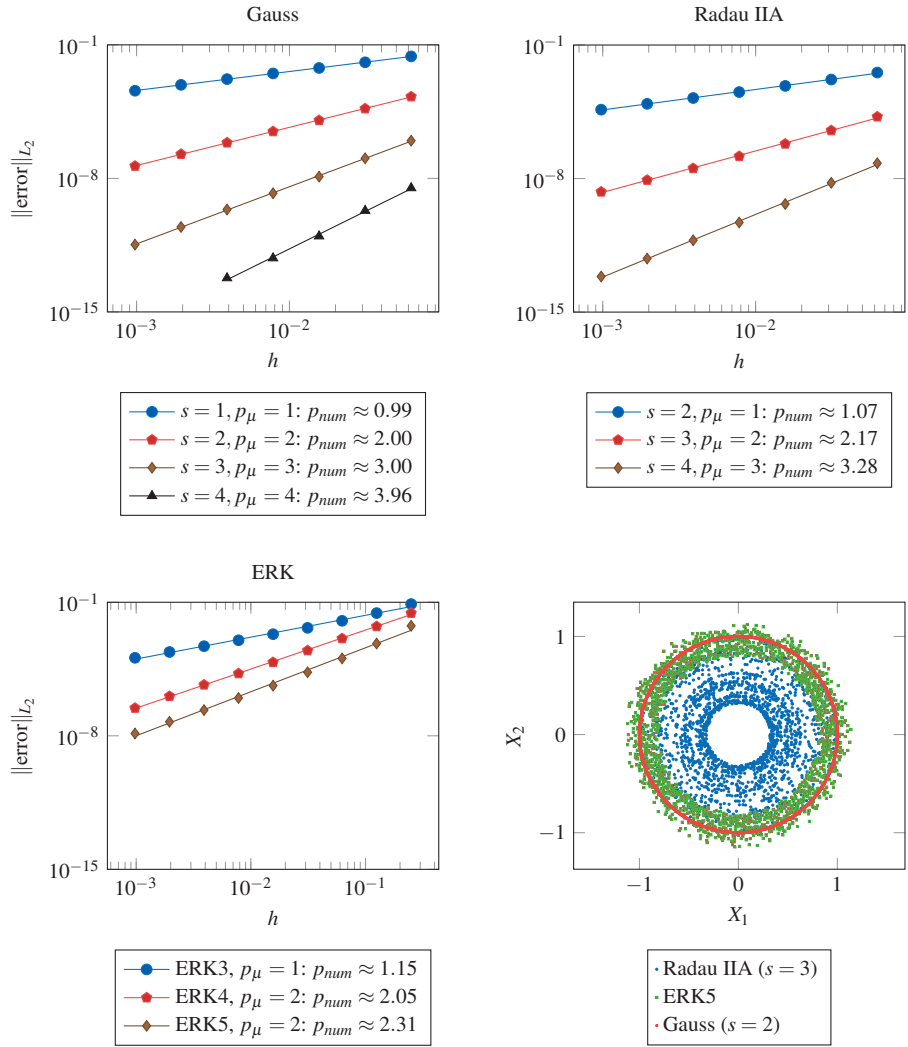
  The next example is the well known Kubo oscillator, a prototype problem for solving oscillatory
  problems, \cite[sec.4.4.1]{milstein04snf}. The SDE is given by
  \begin{equation}
    \dX = \begin{bmatrix} 0 & -a \\ a &0 \end{bmatrix}{X} \dt
    + \begin{bmatrix} 0 & -\sigma \\ \sigma &0 \end{bmatrix}{X} \circ\dW, \qquad
    \label{eq:kubo}
  \end{equation}
  where $a$ and $\sigma$ are real parameters. With $X(0)=[1,0]^T$ this problem has as exact
  solution $X(t)=[\cos{(at+\sigma W(t))}, \sin{(at+\sigma W(t))}]^T$. In our experiments, we
  have used $a=\sigma=1.0$, the mean square error at $t=1.0$ is
  estimated based on $M=1,000$ simulations.

  The Kubo oscillator has the invariant
  $I(X)=X_1^2+X_2^2$.
  To see how well this is preserved by the numerical methods, we have computed one solution path
  by the Gauss ($s=2$) method, one by the Radau IIA ($s=3$) and one by the
  ERK5 method, all with $\pmu=2$. The step size was $h=0.5$, and the integration interval [0,1000].
  From the picture at the bottom right of Figure \ref{fig:kubo} it is clear that the Gauss solution
  stays on the circle given by $I(X(t))=I(X(0))$, the others do not.

\end{example}
\begin{example}
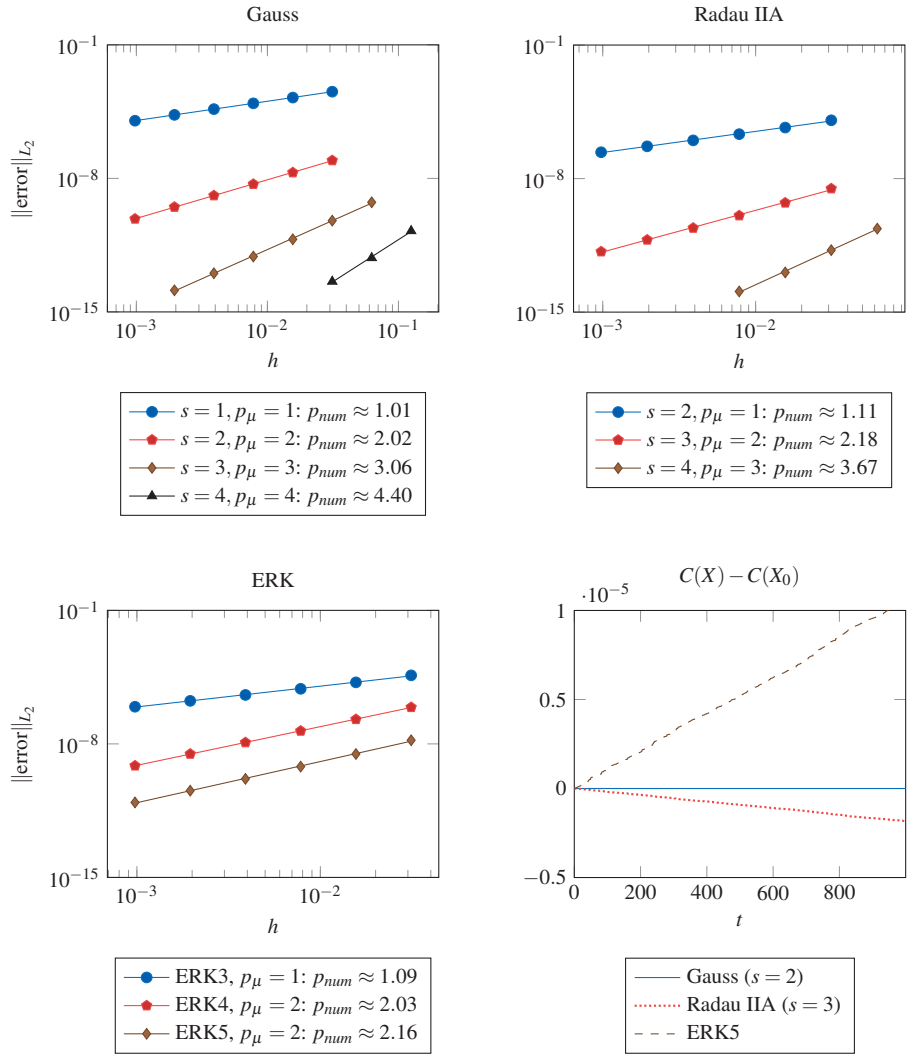
\begin{figure}
  \begin{tikzpicture}[baseline]
  \end{tikzpicture} \hspace*{\fill}
  \begin{tikzpicture}[baseline]
  \end{tikzpicture} \\[5mm]
  \begin{tikzpicture}[baseline]
  \end{tikzpicture}  \hspace*{\fill}
  \begin{tikzpicture}[baseline]
    \end{tikzpicture}
  \caption{Error plots for the stochastic rigid body problem \eqref{eq:rigid_body} (\textit{top and bottom left}).
    Drift of the Casimir $C(X)$ for  Gauss ($s=2$), Radau IIA ($s=3$) and ERK5
    (Fehlberg) using $h=2^{-5}$ for $t\leq 1000$ (\textit{bottom right}).}
  \label{fig:rigid_body}
\end{figure}
This example is based on the deterministic rigid body model from \cite{hairer06gni}. This model has also been used in \cite{cohen14epi} for studying
energy-preserving integrators. The SDE is given by
  \begin{align}
    \dX&=A(X)X\dt + \sigma \,A(X)X\circ\dW \label{eq:rigid_body}
  \end{align}
  with
  \begin{align*}
    A(X)&=\left(
    \begin{array}{ccc}
      0 & X_3/I_3 & -X_2/I_2 \\
      -X_3/I_3 & 0 & X_1/I_1 \\
      X_2/I_2 & -X_1/I_1 & 0
    \end{array}
    \right)
  \end{align*}
  and parameters $I_1=2$, $I_2=1$ and $I_3=2/3$. As initial value we choose
  $X(0) = (\cos(1.1), 0, \sin(1.1))^T$. This problem  conserves the
  invariants
  \begin{align*}
    H(X)&=\frac{1}{2}\left(\frac{X_1^2}{I_1}+\frac{X_2^2}{I_2}+\frac{X_3^2}{I_3}\right),\qquad
    C(X) = X_1^2 + X_2^2 + X_3^2.
    \label{eq:c}
  \end{align*}
  The equation with $\sigma=0.5$ was solved, and the mean square errors at $t=1.0$ based on $M=1,000$
  independent simulations are presented in Figure \ref{fig:rigid_body}. For the Gauss and Radau IIA
  methods, some Newton iterations fail for the larger step sizes. In order to demonstrate the
  conservative properties of the Gauss method, a plot of the Casimir $C(X)$ for one trajectory
  computed by the Gauss ($s=2$) method, the Radau IIA ($s=3$) method and ERK5 has been included.
  As expected, the Gauss method preserves the Casimir, the others do not.
\end{example}

\section{Conclusion}
\label{sec:conclusion}
We have proved that a straightforward extension of deterministic Runge--Kutta methods of order $p$ to
Stratonovich single integrand SDEs results in methods converging with order $\lfloor p/2\rfloor$ in
the mean-square and weak sense. They also inherit certain
properties from their deterministic origin, like preservation of quadratic invariants. These methods are cheaply implementable and
seem to be preferable for solving single integrand problems.

\def\cprime{$'$}

\end{document}